\documentclass[12pt]{article}
\usepackage[utf8]{inputenc}
\usepackage{amsmath}
\usepackage{amsfonts}
\usepackage{amssymb}
\usepackage{epsfig}
\usepackage{epstopdf}
\usepackage{indentfirst}
\usepackage{geometry}
\usepackage{graphicx}
\usepackage{caption}
\usepackage{subfig}
\usepackage[countmax]{subfloat}
\usepackage{rotating}
\usepackage{lscape}
\usepackage[numbers,sort&compress]{natbib}
\RequirePackage[colorlinks=black,citecolor=black,urlcolor=blue]{hyperref}
\usepackage{tabularx}
\usepackage{url}
\usepackage{hyperref}
\usepackage{setspace}
\usepackage{lineno}
\usepackage[section]{placeins}

\DeclareGraphicsExtensions{.jpg,.pdf,.png,.bmp,.eps}
\newtheorem{theorem}{Theorem}[section]

\newtheorem{definition}{Definition}[section]
\newenvironment{proof}{\smallskip\noindent{\bf Proof}}{$\Box$}
\numberwithin{equation}{section}
\def\bbR{{\mathbb R}}

\def\by{{\boldsymbol y}}

\def\bx{{\boldsymbol x}}

\def\bu{{\boldsymbol u}}

\def\cF{{\mathcal F}}

\def\cC{{\mathcal C}}
\def\cD{{\mathcal D}}

\def\cL{{\mathcal L}}
\def\cJ{{\mathcal J}}

\def\cG{{\mathcal G}}

\graphicspath{%
    {converted_graphics/}
    {/}
    {Figures/}
}
\begin{document}

\title{Prediction and estimation of random variables\\ with infinite mean or variance}
\author{Victor de la Pe\~{n}a$^1$, Henryk Gzyl$^2,$\\
  Silvia Mayoral$^3$, Haolin Zou$^1$ and Demissie Alemayehu$^4$\\
$^1$Department of Statistics, Columbia University, New York\\
$^2$Center for Finance, IESA, Caracas\\
$^3$Dept. of Business Administration, Univ. Carlos III de Madrid.\\
$^4$Pfizer Inc, New York, NY 10017\\
vp@stat.columbia.edu, henryk.gzyl@iesa.edu.ve,\\
 smayoral@emp.uc3m.es, hz2574@columbia.edu,\\
 demissie.alemayehu@pfizer.com. \footnote{The authors report there are no competing interests to declare}} 

\date{}
 \maketitle

\baselineskip=1.5 \baselineskip \setlength{\textwidth}{6in}

\begin{abstract}
In this paper we propose an optimal predictor of a random variable that has either an infinite mean or an infinite variance. The method consists of transforming the random variable such that the transformed variable has a finite mean and finite variance. The proposed predictor is a generalized arithmetic mean which is similar to the notion of certainty price in utility theory. Typically, the transformation consists of a parametric family of bijections, in which case the parameter might be chosen to minimize the prediction error in the transformed coordinates. The statistical properties of the estimator of the proposed predictor are studied, and confidence intervals are provided. The performance of the procedure is illustrated using simulated and real data.
\end{abstract}



\section{Introduction and preliminaries}
In this work, we consider the problem of predicting random variables that have an infinite mean or infinite variance. There are several situations where this arises in practice. One example is the prediction of level crossing times (or escape times from unbounded domains) of the $1-$dimensional Brownian motion that is ubiquitous in certain decision-making processes (see Patie (2004) and references therein). It is noted that continuous martingales may be viewed as time changes of $1-$dimensional Brownian motion. Another example is loss estimation when the loss density has heavy tails. This is a rather common problem in the insurance and reinsurance industries and includes the Pareto density with an infinite mean. Heavy tailed $t$ densities that have either infinite mean or infinite variance often come up in applications involving the analysis of social and economic data. Other examples include Cauchy random variables which have an infinite mean, or the random variable, defined on  $([0,1],dt)$ $X(t) = t^{-1/2},$ that has a finite mean but infinite variance with respect to the Lebesgue measure.

Our proposed “best predictor” of a random variable is inspired by Bernoulli’s idea regarding the St. Petersburg bet. Specifically, if $K$ denotes the first time a “tail” appears in the sequence of throws of a fair coin, the payoff of the bet is $X = 2^K.$ If the expected value of the payoff is taken as the price of the bet, then $E[X] = \infty$, and nobody would play this game. An approach to price this bet based on the concept of utility involves the use of a strictly increasing function $u$ defined in an interval containing the range of $X$ such that $u(X)$ has a finite expected value. For applications in economics, finance and risk analysis, $u$ is required to be concave. In our case,  $u(X)$ is assumed to be strictly monotone, having finite mean and variance. The central idea of our approach is that  $u^{-1}\big(E[u(X)]\big)$  can be interpreted as the “best predictor” of $X$ in a different metric. In economics and insurance, this idea is often referred to as  the “certainty equivalent” of $X.$ 

In this paper, a “best predictor” of an unobserved quantity is defined to be a value belonging to a certain class that is closer to the future value of the quantity than any other member of the class. In this formulation, there are at least four issues to consider. The first concerns the nature of the prior information about the unobserved quantity. The second involves the class within which the search is carried out, while the third relates to what is meant by “closer”. Lastly, one needs to address the practical issue of estimating the best predictor from observed data.

When a random variable $X$ is defined on some probability space with a known distribution, the best predictor of $X,$ in the absence of any prior information, is $E[X],$ and the prediction error is $E[(X-E[X])^2].$ Thus, in order to be able to predict and assess the quality of the prediction, the random variables must be at least square integrable.

In our study, we focus on the two related problems of extending the notion of best predictor to random variables with an infinite mean or an infinite variance, as well as the estimation of the best predictor from data. The underlying idea consists of a redefinition of the concept of distance, from a geometric point of view,  over the range of the variable of interest. More specifically, when the random variable has an infinite mean, we can distort its range bijectively in such a way that the transformed random variable has a finite mean and variance. We can use the distortion to induce a new distance in the range of the original random variable, and define a notion of best predictor relative to the new metric. This is where the main contribution of this work lies. 

It is noted that with the choice of coordinates, suitable distances should be defined so that the predictive methods behave accordingly. This issue is familiar to statisticians. For example, when one predicts using Euclidean distance $L_2$ or the $L_1$ distance, one obtains, respectively, the arithmetic mean or the median.

Bernoulli’s idea was extended to the notion of generalized arithmetic mean by Bonferroni (1926) and de Finetti (1931). More recently, it also surfaced in Berger and Casella (1992) and de Carvalho (2016). The use of changes of coordinates, also known as transformation theory in regression analysis and other applications is well documented in the literature (see, e.g., Bartlett (1947), Tukey (1957), Box and Cox (1964), Doksum (1984), Osborne (2002), Yang (2006) and Rojas-Perilla (2018)). Another line of work is the proposal of Cirillo and Taleb (2016), which concerns the tail of a distribution when the mean is infinite.

The rest of the paper is organized as follows.  In Section 2 we show that the certainty equivalent or generalized arithmetic mean is the best predictor in a metric defined by the change of variables. We consider the cases of unrestricted and restricted predictors. The latter comprises linear and nonlinear regression. Restricted prediction coupled with the change of variable methodology provides the nexus between our proposal and the transformation methodology. In Section 3, we examine the statistical properties of the predictor. Section 4 is devoted to illustrative examples, including analytical calculations and real and simulated data results. For selected random variables and transformations, it is shown that  the best estimators and the estimation error can be explicitly calculated. Simulation studies are performed to support the analytical calculations. Examples with real data illustrate the approach in the case of a Pareto density with an infinite mean and a Student$-t$ density with an infinite variance. For the case of finite mean but infinite variance, we consider a change of variables that can be interpreted as an inversion with respect to a  $1$-dimensional sphere, and for which the best predictor is the harmonic mean. Lastly, in Section 5 we briefly discuss the question of how to choose the change of variables and conclude with a few remarks.

\section{Predictors in non-Euclidean metrics}
In this section, we consider the problem of constructing predictors when the class of predictors is unrestricted as well as when the class is restricted. 

\subsection{Unrestricted prediction}

Consider a complete probability space $(\Omega,\cF, P)$ and a continuous, strictly monotone function  $u: \cJ\to\bbR,$ where $\cJ$ is some bounded or unbounded interval on the real line. 

Let  
\begin{equation}\label{m1}
\cL^p_u = \{X:\Omega \to \cJ| \big(E[(u(X))^p]\big)^{1/p}< \infty\}.
\end{equation}
Define the distance
\begin{equation}\label{m2}
d_u(X,Y) = \big(E[(u(X)-u(Y))^2]\big)^{1/2} \;\;\;\mbox{ for any}\;\;\;X,Y \in\cL_u^2.
\end{equation}

When $u = id : \cJ \to \cJ$ is the identity mapping, $d(X,Y)$ will denote the associated distance. For any stricly monotone function $u$ up to a null set, $d_u(X, Y)$ is a distance on $L^2_u$. 

We now state two results, the proofs of which are well established and may be found in any standard text book (see, e.g.,  Berger and Casella (1992) or Carvalho (2016)).

\begin{theorem}\label{m3}
Let $X$ be a $\cJ$ valued random variable and suppose that $X\in\cL_u^2.$ Then there exists a number $m_u\in\cJ$ such that $m_u=argmin\{E[\big(u(X)-u(x)\big)^2] : x\in\cJ\}.$  It is given by
\begin{equation}\label{m4}
m_u \equiv E_u[X] = u^{-1}\big(E[u(X)]\big).
\end{equation}
The square of the prediction error being:
\begin{equation}\label{m5}
E[\big(u(X) - u(m_u)\big)^2] = E[\big(u(X) - E[u(X)]\big)^2] =\sigma^2(u(X)).
\end{equation}
\end{theorem}
Note that because $u(J )$ is an interval, $E[u(X)]\in u(\cJ)$ and then $m_u=E_u[X]\in\cJ..$  We will refer to it as the $u-$mean value of $X.$

When there is prior information contained in a sub-$\sigma-$algebra $\cG \subset \cF,$ in analogy to the Euclidean norm case, we have the following result.

\begin{theorem}\label{m6}
Let $u:\cJ\to\bbR$ be continuous and strictly monotone,  $Y\in\cL_u^2$ and $\cG$ be a sub-$\sigma$-algebra of $\cF.$
Then, there exists a square-integrable, $\cG-$measurable random variable, denoted by $E_u[Y|\cG]$ and computed by
\begin{equation}\label{m7}
E_u[Y|\cG] = u^{-1}\big(E[u(Y)|\cG]\big)
\end{equation} such that
$$E_u[Y|\cG] = \rm{arginf}\{E[\big(u(Y) -u(X)\big)^2]|Y\in\cL_u^2\;\;\mbox{and}\;\;X\in\cG\}.$$
\end{theorem}
This result explains what we mean by prediction in the $d_u-$distance.

\subsection{Restricted prediction}
As in Section 2.1, consider the probability space $(\Omega,\cF,P)$ with $\cF$ containing all $P-$null sets, and $\cJ-$valued random variables. In many applications, it is of interest to consider a class $\cC$ of predictors smaller than $L_2(\sigma(X),P)$ as in the case of linear or non-linear regression. It is assumed that $\cC\subset L_2(\sigma(X),P)$ is closed in the $L_2(P)$ norm. As an example consider  $\cC=\{\sum_{k=1}^na_k\phi_k(X)|a_k\in\bbR\}$ where $\{\phi_k:k=1,...,n\}$ is some finite collection of measurable functions defined on $\cJ.$ A common exercise is, given a $Y\in L_2(\cF,P)$, to verify that there is a point, denoted here by $\Pi_{\cC}(Y)$  that solves the problem of finding:
$$\Pi_{\cC}(Y) = \rm{arginf}\{E[(Y-\phi(X))^2| \phi(X)\in\cC\}.$$

We shall call this point the optimal $\cC$-predictor of $Y.$

The objective then is to establish how the best predictor changes when we change the coordinates and the distance in the range of the random variables of interest. If the class of predictors satisfies an invariance  condition, we can relate the prediction in the $d_u$ distance to the prediction in the standard $L_2$ distance between random variables as follows.

\begin{theorem}\label{rest1}
With the notations introduced above, suppose that $\cC$ satisfies the following congruence condition with respect to the change of variables $u$:
\begin{equation}\label{consis}
u\circ\cC\circ u^{-1} = \cC\;\;\;\;\Leftrightarrow \;\;\; u^{-1}\circ\cC\circ u = \cC.
\end{equation}
Denote by $\Pi_{\cC,u}(Y)$ the best predictor of $Y$ by an element from $\cC$ in the $d_u$ distance. And denote by $\Pi_{\cC}(u(Y))$ the best predictor of $u(Y)$ by an element from $\cC$ in the standard distance. Then
\begin{equation}\label{rest2}
\Pi_{\cC,u}(Y) = u^{-1}\bigg(\Pi_{\cC}(u(Y))\bigg).
\end{equation}
\end{theorem} 

\begin{proof}
Notice that
$$\Pi_{\cC,u}(Y) = \rm{arginf}\{d^2_u(Y,\phi(X))|\phi\in\cC\} = \rm{arginf}\{d^2(u(Y),u(\phi(X)))|\phi\in\cC\} $$
$$= \rm{arginf}\{d^2(u(Y),u(\phi(u^{-1}\big(u(X)\big))))|\phi\in\cC\} = \rm{arginf}\{d^2(u(Y),\phi(u(X)))|\phi\in\cC\}$$
$$ = u^{-1}\Pi_{\cC}(u(Y)).$$
It is in the second to the last step where \eqref{consis} was used.
\end{proof}

As an illustration, consider $\cC=\{\sum_{k=1}^na_k\phi_k(X)|a_k\in\bbR\}$, where the collection $\{\phi_k:k=1,...,n\}$ and the set of real numbers $\{a_1,...,a_n\}$ are such that $\cC$ is a closed subset of $L_2(\Omega,\sigma(X),P).$ Then $u^{-1}(\left(\sum_{k=1}^n a^*\phi(u(X))\right)$ is the best predictor of $Y$ in the $d_u-$distance.

\section{Estimating the best predictor and statistical properties of the estimator}
\subsection{Estimation of the best predictor}
Suppose $\{X_1,...,X_n\}$ is a random sample from the distribution of $X.$ Then the quantity that is closest to all of them in the $u-$metric is:
\begin{equation}\label{m10}
\overline{X}_u(n) = u^{-1}\bigg(\frac{1}{n}\sum_{k=1}^nu(x_i)\bigg).
\end{equation}
            Clearly, the average is in $u(\cJ)$ and therefore $\overline{X}_u(n)\in\cJ.$  Applying the strong law of large numbers, we get
            $$\overline{X}_u(n)\, \underrightarrow{a.s.}\, E_u[X]=u^{-1}E[u(X)] \equiv m_u.$$

Note that  
$$E[u(X)]\in u(\cJ)\;\;\mbox{ therefore}\;\; m_u=u^{-1}\big(E[u(X)]\big)\in \cJ.$$
 
 Thus $ \overline{X}_u(n)$ can be thought of as an estimator of $X$ despite the fact that $E[X] = \infty.$ Further, the average, minimal square $u-$distance is
\begin{equation}\label{m11}
\frac{1}{n}\sum_{k=1}^nd_u(x_k,\overline{X}_u(n))^2 = \frac{1}{n}\sum_{k=1}^n\bigg(u(x_k)-\frac{1}{n}\sum_{k=1}^nu(x_i)\bigg)^2.
\end{equation}

 This idea is used to define the sample variance as
 
 \begin{equation}\label{m12}
\overline{\sigma}^{2}(u(X)) = \frac{1}{n}\sum_{k=1}^n d_u\big(X_k,E_u(X)\big)^2 = \frac{1}{n}\sum_{k=1}^n\big(u(X_k)-E[u(X)]\big)^2.
\end{equation}
    
 And, the unbiased estimator of the variance as:
 \begin{equation}\label{m13}
\hat{\sigma}^2(u(X)) = \frac{1}{n-1}\sum_{k=1}^n\bigg(u(X_k) - u(\overline{X}_u(n))\bigg)^2.
\end{equation}

\subsection{Properties of the estimator}
Under the above assumptions, let $\bx = {x_1,...,x_n}$ be a random sample from the distribution of $X.$ Then $u(x) = \{u(x_1),...,u(x_n)\}$ is a random sample corresponding to $u(X).$ By \eqref{m10}, $\overline{X}_u(n)$  is the best predictor of an element of the set $\bx$ in the $d_u-$distance. Theorem \ref{BE1} describes the basic properties of the sample estimator of $E_u[X].$

\begin{theorem}\label{BE1}
Let us now suppose that $u$ and $u^{-1}$ are both continuously differentiable. Then, the sample mean $\overline{X}_u(n)$ given by (\ref{m10}) is an $E_u$ unbiased and consistent estimator of $E_u(X),$ and $\hat{\sigma}^2(u(X))$ defined in \eqref{m13}, is an unbiased estimator of $Var(u(X))^2$.
\end{theorem}

\begin{proof} Let $v=u^{-1}.$ Note that 
$$
E_u\big[v\big(\frac{1}{n}\sum_{k=1}^nu(X_k)\big)\big] = v\bigg(\frac{1}{n}\sum_{k=1}^nEu(X_k)\bigg) = v\big(E[u(X)]\big)=E_u[X].
$$
That is $\overline{X}_u(n)$ is $E_u$ unbiased. It is easy to verify that 
$$d_u^2\big(\overline{X}_u(n),E_u[X]\big) = \frac{1}{n}Var(u(X)).$$
In order to verify the consistency of the estimator, we should consider the usual quadratic distance. 
$$\begin{aligned}
E\left[\bigg(\overline{X}_u(n)-E_u[X]\bigg)^2\right] = E\left[\bigg(v\big(\frac{1}{n}\sum_{k=1}^nu(X_k)\big)-v\big(E[u(X)]\big)\bigg)^2\right] \\
\approx v'\big(E[u(X)]\big)^2E\left[\bigg(\frac{1}{n}\sum_{k=1}^nu(X_k)-E[u(X)]\bigg)^2\right]=v'\big(E[u(X)]\big)^2\frac{1}{n}Var(u(X)).\end{aligned}
$$
The second identity follows from a first order expansion of $v$ about $E[u(X)].$ Using this and the Markov-Tchebyshev inequality we obtain that for any $a>0:$
$$P\left(\left|\overline{X}_u(n)-E_u[X]\right| > a\right] \rightarrow 0\;\;\;\mbox{as}\;\;\;n \to \infty.$$
That is $\overline{X}_u(n)$ is a consistent estimator of $E_u[X].$ 
\end{proof}

If the mean $E[X]$ is infinite, consider a modified coordinate system in which the mean is finite, compute the mean and bring it back to the original space by inverting the coordinate transformation.

\subsection{Confidence interval for the empirical estimator of the mean}

It is noted that while the empirical mean  $\frac{1}{n}\sum u(X_k)$ is an average of i.i.d. random variables, the same cannot be said about $\overline{X}_u(n)=u^{-1}\big(\frac{1}{n}\sum u(X_k)\big).$  However, as indicated earlier the convergence of $\frac{1}{n}\sum u(X_k)$ to $E[u(X)]$ can be transformed into a convergence of $\overline{X}_u(n)$ to $E_u[X].$ 

Similarly, to determine a confidence interval for the predictor, one can first do it in the transformed coordinates, then map the confidence interval back onto a confidence interval in the original coordinates. This is possible because the function is monotone. This is what we do in the numerical examples below.

So, suppose a $q_\alpha$ has been determined such that
\begin{equation}\label{CI1}
P\bigg(\left|\frac{\frac{1}{n}\sum_{k=1}^n u(X_k)- E[u(X)]}{\hat{\sigma}(u(X))}\right|> q_\alpha\bigg) \leq 1-\alpha.
\end{equation}
That such $q_\alpha$ exists is clear from the fact that the absolute value within the round brackets is positive with probability $1,$ and finite, that is, the probability of the event goes down to $0$ as $q_\alpha$ increases. To establish a confidence interval for $\frac{1}{n}\sum_{k=1}^n u(X_k)$ we must consider the fact that $u(\cJ)$ could be a bounded interval. Write $-\infty\leq A=\inf u(\cJ)<\sup u(\cJ)=B\leq\infty,$ and put
\begin{equation}\label{endpts}
l_\alpha(n) = \max\big(A,\overline{X}_u(n)-q_\alpha\hat{\sigma}(u(X))\big),\;\;\;\mbox{and}\;\;\;u_\alpha(n)=\min\big(B,\overline{X}_u(n)+q_\alpha\hat{\sigma}(u(X))\big).
\end{equation}

Then clearly (\ref{CI1}) can be expressed as
$$ P\bigg(E_u[X]\,\in \big(l_\alpha(n), u_\alpha(n)\big)\bigg) = 1-\alpha.$$
We summarize these comments in the following statement. We suppose that $u$ is strictly decreasing but the argument applies when it is increasing.

\begin{theorem}\label{CI3}
Let $\{X_1,....,X_n\}$ be a sample of a positive random variable that has an infinite mean. Let $u:\cJ=[0,\infty)\to u(\cJ)\subset [0,\infty)$ be a continuous monotone decreasing function such that $E[u^2(X)]<\infty.$ Suppose that a number $q_\alpha$ has been determined such that in the transformed variables (\ref{CI1}) holds true. With the notations introduced above, the interval $\big(l_\alpha(n), u_\alpha(n)\big)$ is transformed back to the original coordinates as:
$$L_\alpha(n)=v\big(u_\alpha\big),\;\;\mbox{and}\;\;U_\alpha(n)=v\big(\overline{X}_u(n)-q_\alpha\hat{\sigma}(u(X))\big),$$
and (\ref{CI1}) becomes
\begin{equation}\label{CI2}
P\bigg(E_u[X]) \in [L(n),U(n)]\bigg) \leq 1-\alpha.
\end{equation}
\end{theorem}

\section{Analytical and numerical examples}
In this section, we illustrate computational aspects of the proposed predictor and its estimator using analytical examples as well as simulated and real data. In the examples that involve crossing times, we consider two parametric transformations and show how to choose the parameter in such a way that the variance of the transformed variable is minimal.

\subsection{The median as an $L_2-$predictor in a natural change of variables}
Consider a continuous random variable $X$ with an infinite mean, and let $F$ denote its cumulative distribution function. Consider a continuous random variable $X$ and let $F$ denote its cumulative distribution function.  Let $u(x)=F(x),$ then:
$$E[u(X)]=\int F(x)dF(x) = \frac{1}{2}.\;\;\;\;\Longrightarrow\;\;\;\;E_u[X] = F^{-1}(\frac{1}{2}).$$

\subsection{Barrier crossing times}\label{ex2}
For a Brownian motion process in which a decision is made once a barrier is reached, the expected time to reach the barrier is infinite. Let $X(t)$ denote the standard Brownian motion.  If $L>0,$ and $\tau_L=\inf\{t>0: X(t)\geq L\}$ denotes the time it takes to reach level $L,$ then it is known that:

\begin{equation}\label{bm2}
f_{\tau_L}(t) = \frac{L}{(2\pi t^3)^{1/2}}e^{-L^2/2t}, \;\;\;t>0.
\end{equation}

Next, let $u(t)=\exp(-b/(2t))$, so that $u^{-1}(s) = b/(2\ln s^{-1}).$ Then, using (\ref{bm2}) 
 $E[\exp(-b/(2\tau_L))]=L/(\sqrt{b+L^2}).$ Therefore the best $u-$predictor is:
 
$$E_u[\tau_L]= \frac{b}{2\big(\ln\sqrt{(b + L^2)} -\ln L\big)} = \frac{b}{2\big(\ln\sqrt{(\frac{b}{L^2}+1)}\big)}=\frac{b}{\ln(\frac{b}{L^2}+1)}.$$

As $b\to 0,$ $E_u[\tau_L]\to L^2.$ The prediction error is:
$$Var(u(\tau_L)) = \frac{L}{\sqrt{2b + L^2} }- \bigg(\frac{L}{\sqrt{b + L^2}}\bigg)^2$$

Let $\xi=b/L^2.$ The minimum of $Var(u(\tau_L))$ is reached at the solution to $(\xi+1)^4-(2\xi+1)^3=0.$ It is given by $\xi^*\approx 5.2223$ at which the prediction error is $\sqrt{Var}_{max}\approx 0.135.$

To consider a different change of variable,  let $u(t)=\exp(-ct).$ The inverse of this strictly monotone function is $u^{-1}(s)=\frac{1}{c}\ln(\frac{1}{s}).$ It is also known that $E[e^{-c\tau_L}]=e^{-\sqrt{2c}L}.$ This implies that the best $u-$predictor in this case is:
$$E_u[\tau_L] = \sqrt{\frac{2}{c}}L.$$
The prediction error is given by 
$$Var(u(\tau_L))=e^{-2\sqrt{c}L} - e^{-2\sqrt{2c}L}.$$

This error is minimal at $\sqrt{c}=\ln(2)(\sqrt{2}+1)/4L,$ and equals $Var_{min}\approx0.1269.$ It is noted that while the optimal value of $c$ depends on $L$, the minimal value does not.

\subsection{A Pareto variable with infinite mean}
This case might be of potential application in risk management because the coomon Pareto distribution has heavy tails.The specific density that we consider is 
$$f_X(x)=\alpha(1+x)^{-(1+\alpha)},\;\;\; \mbox{with}\;\;\; 0<\alpha<1.$$

In the following illustrations, we first analytically determine the prediction error for the transformed variables. Then we carry out a simulation study of the predictor, the prediction error, and the confidence interval. Lastly, we will give an example using real data from a Pareto density with an infinite mean.

Our aim is to estimate the mean and the prediction error. Consider the family of bijections $u_b(x)=(1+x)^{-b}$ with $b\geq 1$ that map $[0,\infty)$ onto $(0,1].$  Since $E[u_b(X)] = \frac{\alpha}{b+\alpha},$ the best $u_b-$predictor  and its variance are:

\begin{equation}\label{tmv1}
E_{u_b} = \bigg(\frac{b+\alpha}{\alpha}\bigg)^{1/b}-1,\;\;\;\mbox{and}\;\;\;Var(u(X)) = \frac{\alpha b^2}{(2b+\alpha)(b+\alpha)^2}.
\end{equation}

Both quantities are finite for each $b\geq 1,$ and the estimation error decreases as $b \to \infty.$  
Next, we generate a sample $\{s_k:k=1,...,n\}$ from $U[0,1]$ from which we obtain a sample $\{x_k:k=1,...,n\}$ of a Pareto($\alpha$)  by setting $x_k=(1-s_k)^{1/\alpha}-1$ (with $\alpha=1/2$). Then compute $u_b(x_k)$ to which we apply the routine described above. That is, we compute the empirical variance $\hat{\sigma}^2(u_b(X))$ for a sample of size $n=100,500, 1000, 5000.$ The result is displayed in Figure \ref{fig3}, in which we plot the variance as function of $b.$.

\begin{figure}[h]
\centering
\subfloat[$\sigma(u_b(X))$ for each sample size]{\includegraphics[width=3.0in,height=3.0in]{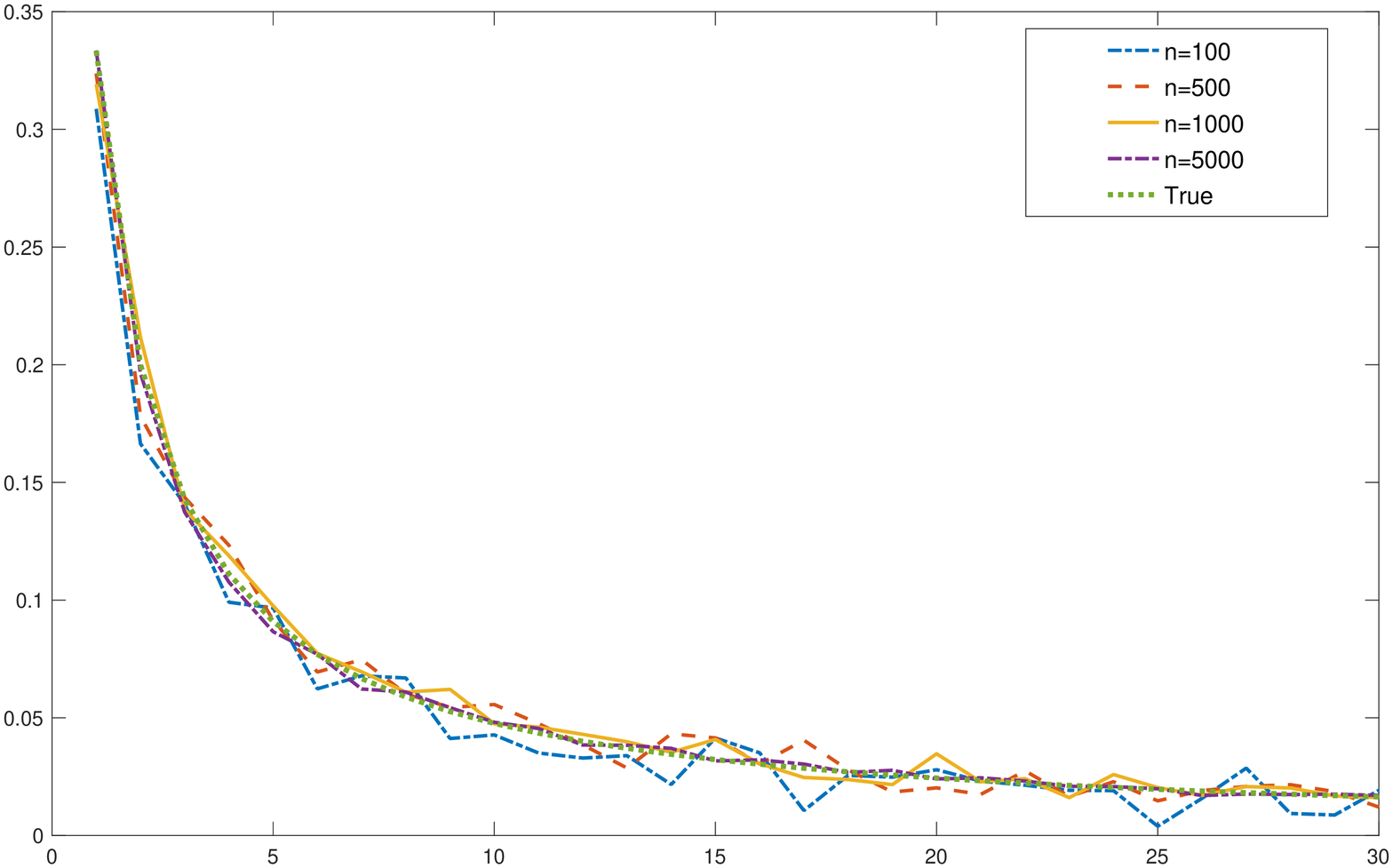}\label{fig3}}
\subfloat[Averaged $Var(u_b(X))$]{\includegraphics[width=3.0in,height=3.0in]{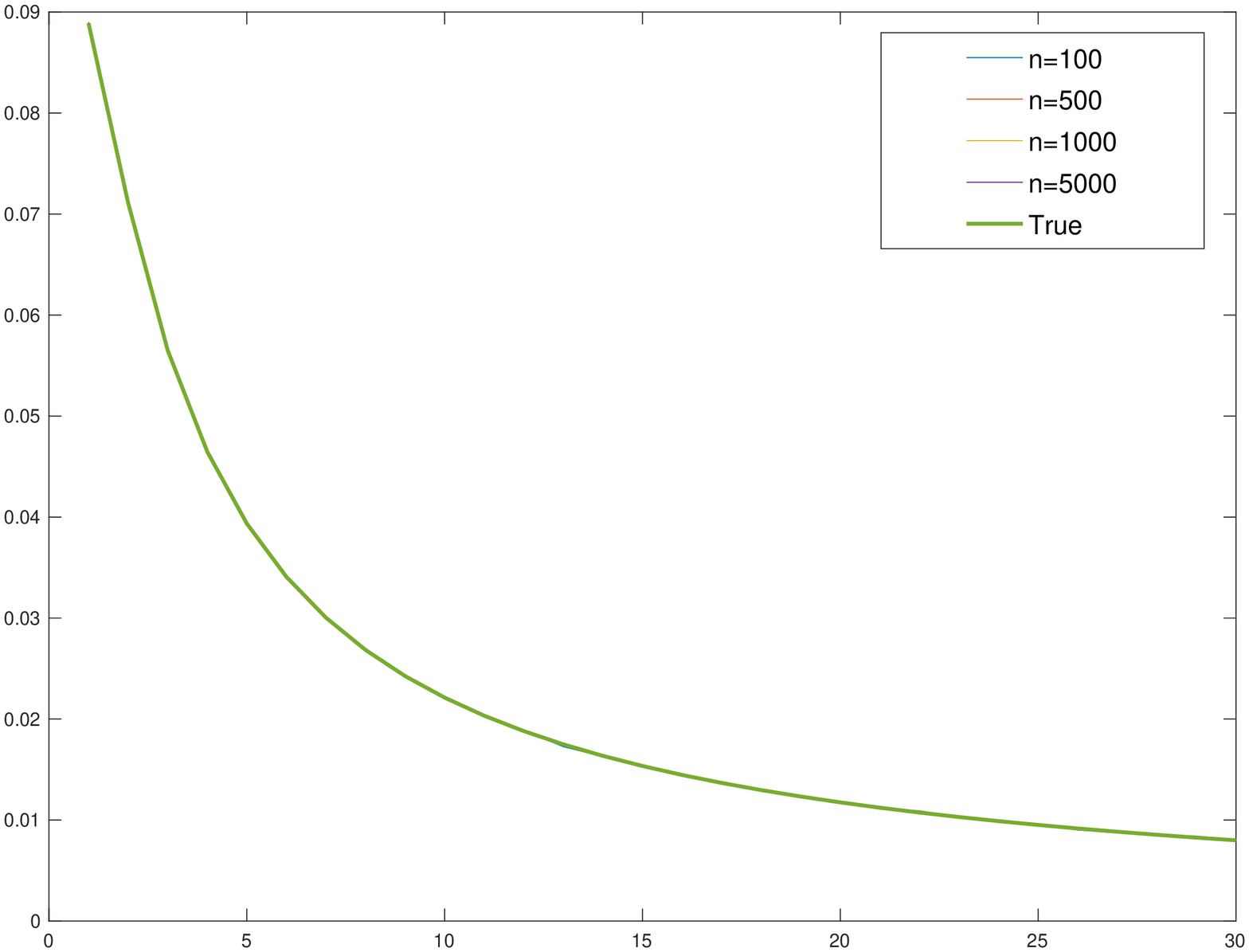}\label{fig4}}
\caption{$Var(u_b(X))$ as function of the parameter $b$}
\end{figure}

In addition, we simulated $5000$ samples of the different sizes and computed the average of the empirical variances of the different sample sizes. The result is displayed in Figure \ref{fig4}. As expected, the variance decreasses with $b.$

Given a sample $\{x_1,...,x_n\}$ of this distribution, the sample mean in the $u_b-$distance is:
$$\overline{X}_{u_b}(n) = \left(\frac{1}{n}\sum_{k=1}^N\frac{1}{(1+x_k)^b}\right)^{-1/b} - 1.$$

In Fig \ref{fig5}, we display the result of the computation for the samples of sizes $100, 500, 1000, 5000,$ and present the empirical mean as a function of $b.$ To smooth the result, we simulated $5000$ samples and averaged the mean for each $b.$ The result is displayed in Figure \ref{fig6}. Our predictor matches the theoretical value, and $\overline{X}_{u_b}(n)\to 0$ as $b\to\infty.$ 

\begin{figure}[h]
\centering
\subfloat[One sample of each size]{\includegraphics[width=3.0in,height=3.0in]{Figures/FUNCION_MEDIA_caso_varianza__media_inf_1trayectoria.eps}\label{fig5}} 
\subfloat[Averaged over samples]{\includegraphics[width=3.0in,height=3.0in]{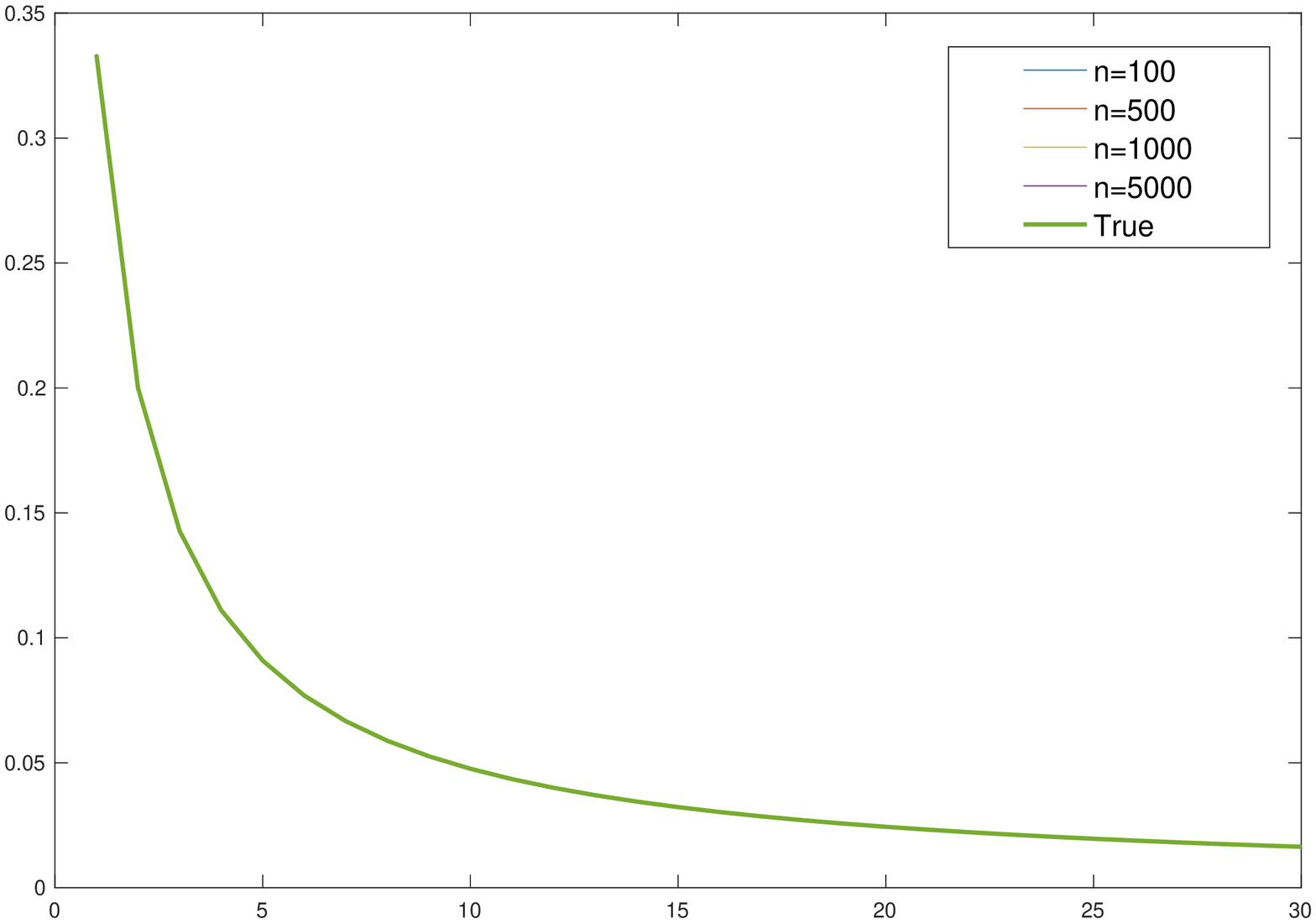}\label{fig6}}
\caption{Empirical best predictor of the mean as a function of $b$}
\end{figure}

\begin{figure}[h]
\centering
    \subfloat[Transformed coordinates]{\includegraphics[width=3.0in,height=3.0in]{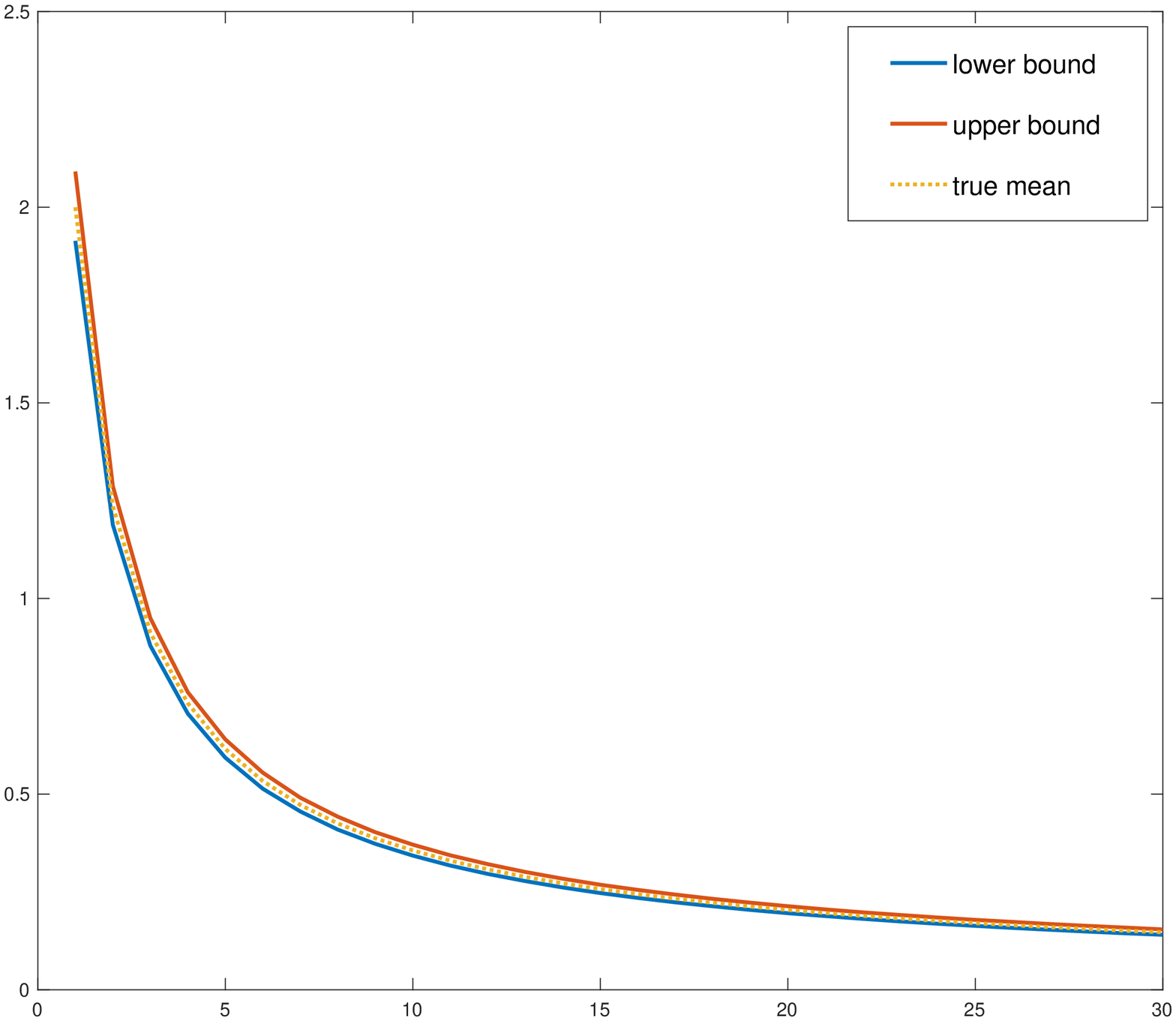}\label{fig9}} 
       \subfloat[Original coordinates]{\includegraphics[width=3.0in,height=3.0in]{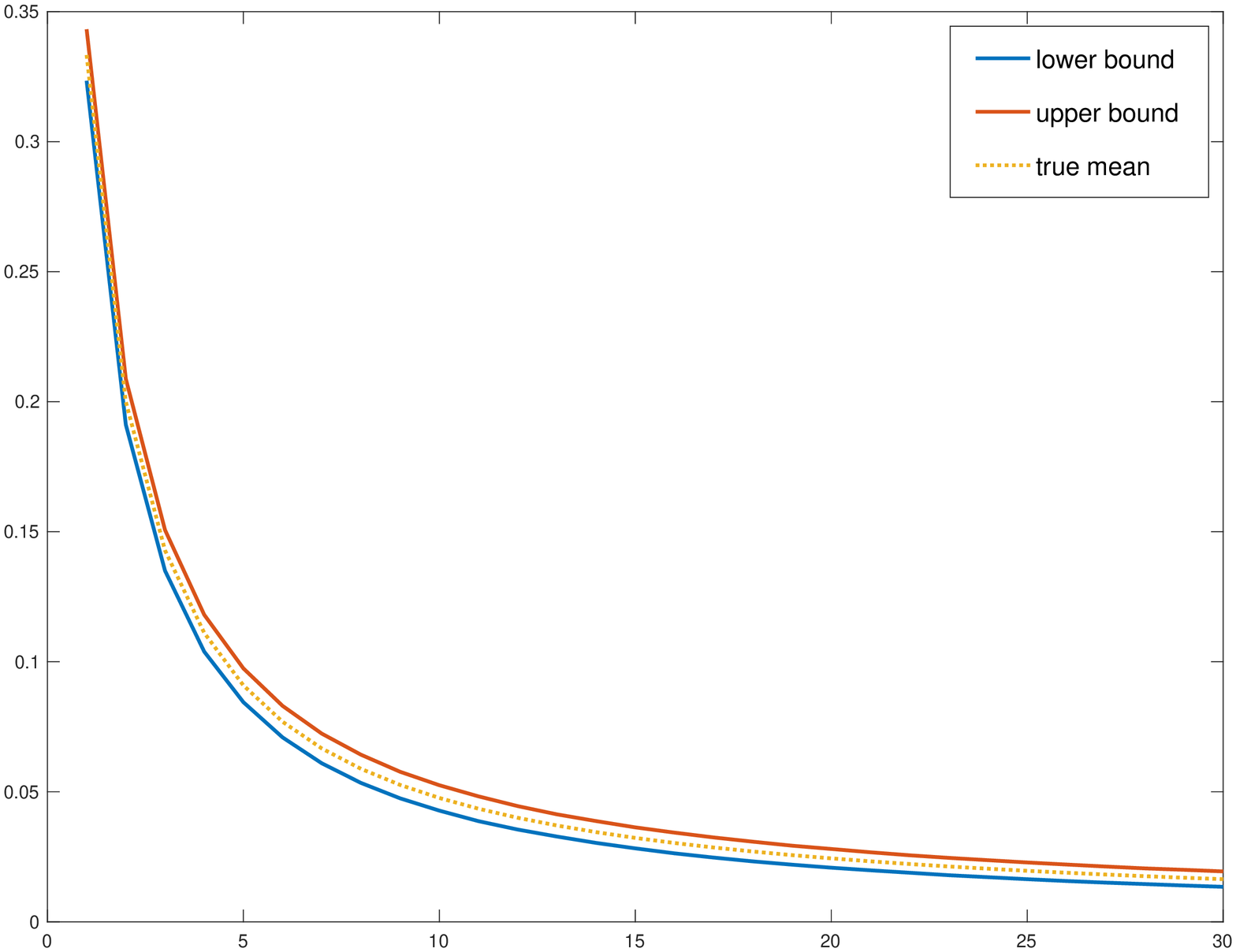}\label{fig10}}
\caption{Confidence intervals as function of $b$ for the Pareto variable}
\end{figure}

Finally, we illustrate construction of confidence intervals for the predictor of the mean, using 
 a sample of size n = 5000. For this, we compute the confidence intervals for the estimator of the predictor in the new coordinates in which the standard limit theorems can be applied, and then, transform back to the original coordinate system. The results are displayed in Figures \ref{fig9} and \ref{fig10}. The left and right panels display the plots of the confidence interval as a function of the parameter $b$ in the transformed and original coordinates, respectively.
 
The variance and the confidence intervals a decrease as function of $b,$ showing that there is no optimal parameter as in example \ref{ex2}. Nevertheless, we can use \eqref{tmv1} to choose a suitable $b$ that yields an acceptable variance and confidence interval, corresponding to the chosen sample size.

\subsection{Empirical examples using real-world data}
The application of the proposed approach to risk modeling and in economics was further illustrated with real-world data on historical hurricane economic loss and US household wealth.,  In both cases a parametric density with an infinite mean was estimated.  

For the hurricane economic loss application, we used normalized hurricane economic loss data(CL05) for the period 1900-2005   (Pielke et al.(2008)). There are in total 207 recorded hurricanes. As shown in Figure \ref{fludata} while most of the losses  (i.e., $180$ out of $207$) were under \$10 billion, in some cases the damage was as high as \$139.5 billion (as was the loss that resulted from the notorious Great Miami hurricane of 1926), indicating a heavy tail.

\begin{figure}[h] 
  \centering
  \includegraphics[width=3.0in,height=3.0in,keepaspectratio]{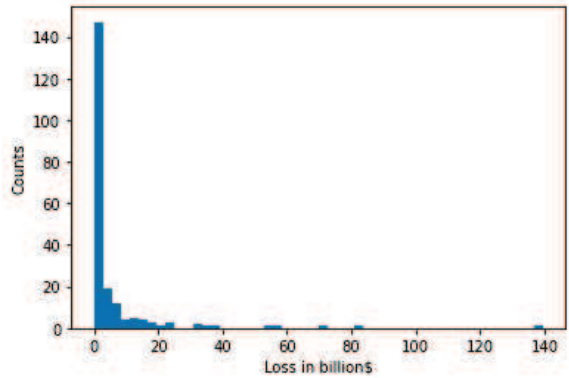}
  \caption{Hurricane losses}
  \label{fludata}
\end{figure}

To evaluate the central tendency of the losses, a shifted Pareto family with density $f_X (x) = \alpha(1 + x)^{-(1+\alpha)}$  was fitted. The MLE of $\alpha$ is $1.046,$ so the distribution has finite expectation but infinite variance. As above, we use the parametric family of transformations $u_b(x) = (1+x)^{-b}.$ The $95\%$ confidence intervals  for $\bar{X}_u$ in the transformed and original coordinates are shown in Figures \ref{fig13}-\ref{fig14}.
	
\begin{figure}[h]
		\centering
  \subfloat[Transformed coordinates]{\includegraphics[width=3.0in,height=3.0in]{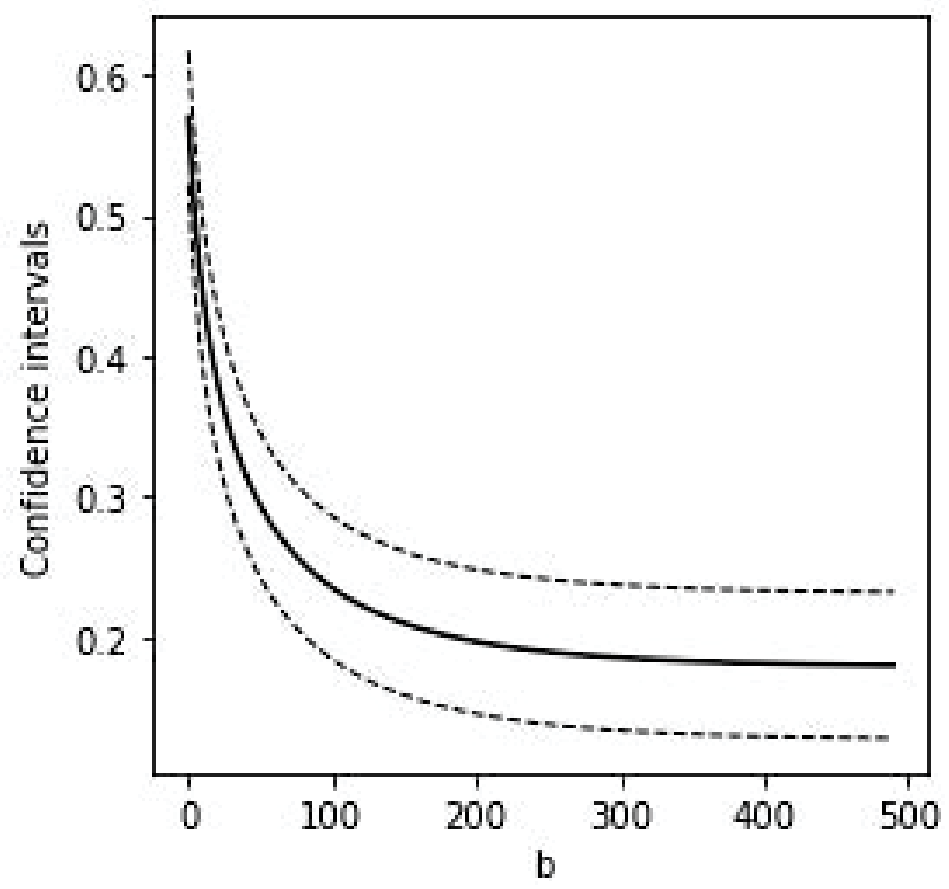}\label{fig13}} 
  \subfloat[Original coordinates]{\includegraphics[width=3.0in,height=3.0in]{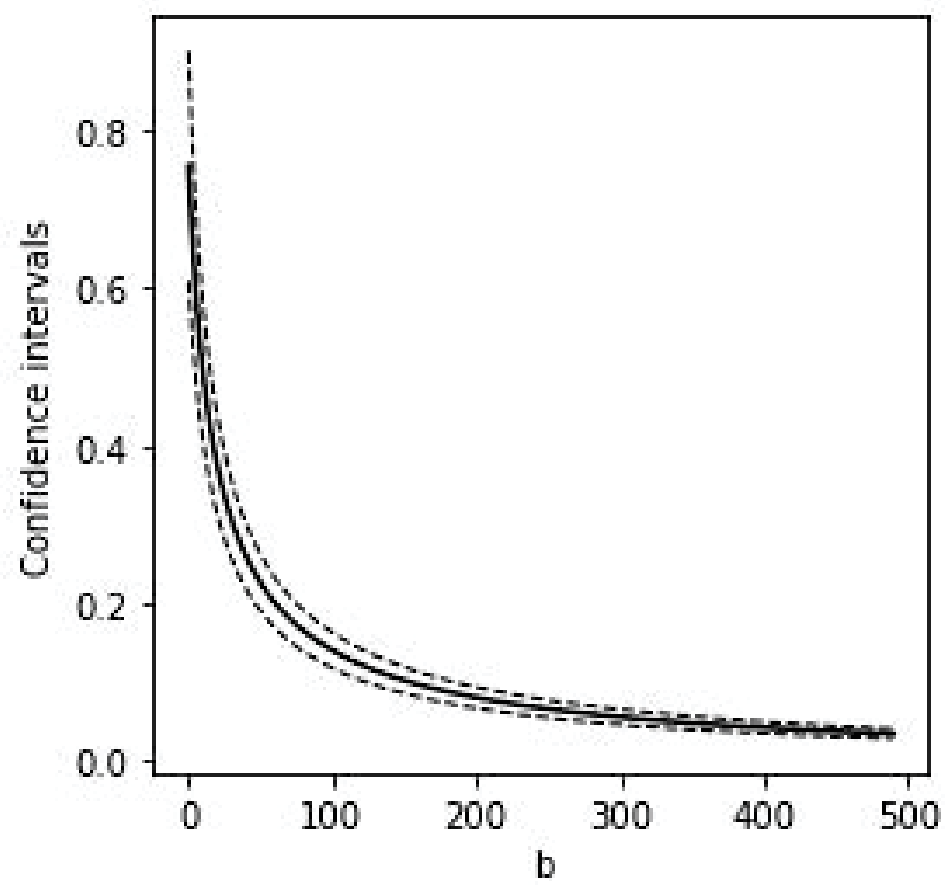}\label{fig14}}
\caption{Confidence intervals as function of $b$ for the Pareto variable}
\end{figure}

For the US household wealth example, we used data from the 2019 Survey of Consumer Finance(SCF) (\cite{SCF1}). SCF is normally a triennial cross-sectional survey of U.S. families, and includes information on families' balance sheets, pensions, income, and demographic characteristics. Additional information is also included from related surveys of pension providers and earlier such surveys conducted by the Federal Reserve Board. No other study for the country collects comparable information. The study was sponsored by the Federal Reserve Board in cooperation with the Department of the Treasury. ((\cite{SCF1}).
	
There were 5,777 respondents in the SCF2019 dataset, but each entry was imputed 5 times to keep the confidentiality of respondents’ identities. For the purpose of our study, we took the average of the 5 imputations to represent the original data. We are interested in the ’average’ household wealth. We truncated the data, so that the appropriate transformations can be applied. Figure \ref{HWD} is a histogram of $\log(wealth)$ which suggests a Student-t distribution. The estimated paramteres are, location: $12.74,$  scale:$2.81,$ and degrees of freedom: $14.$ If we denote the wealth by $X,$ then $X=e^Y,$ where $Y$ is the Student -t just described. Therefore, the moments of $X$ are all infinite. 
	
\begin{figure}[h]
\centering
      \includegraphics[width=2.5in,height=2.5in]{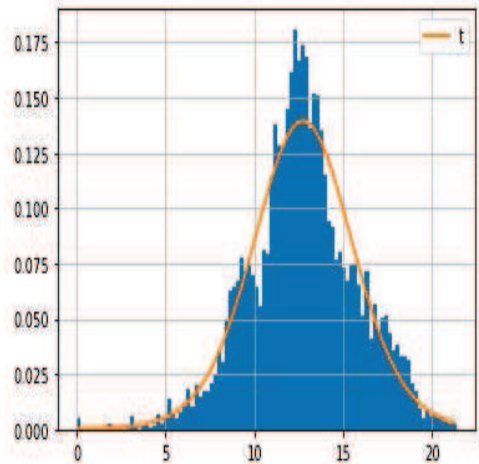}
\caption{Histogram of US household wealth in SCF2019}
\label{HWD}		
\end{figure}
	
To obtain an estimator using our proposal, we consider two coordinate transformations: A non-parametric transformation given by $u(x) = \ln x,$ and a parametric one given by $u_b(x) = x^b$ (applied to log wealth). These transformations are similar to a choice of utility function. We applied the proposed procedure and obtained the results displayed in Figures 7a- 7b. The left panel describes the confidence interval as a function of the parameter $b$ in the transformed scale, whereas the right panel describes the same in the original scale. The red and blue lines depict the upper and lower confidence intervals, respectively. The straight lines depict the confidence interval associated with $u(x) = \ln x,$  while the curved lines correspond to $u_b(x) = x^b.$  	
	
	\begin{figure}[h]
	\centering
	\subfloat[Transformed coordinates]{\includegraphics[width=2.5in,height=2.5in]{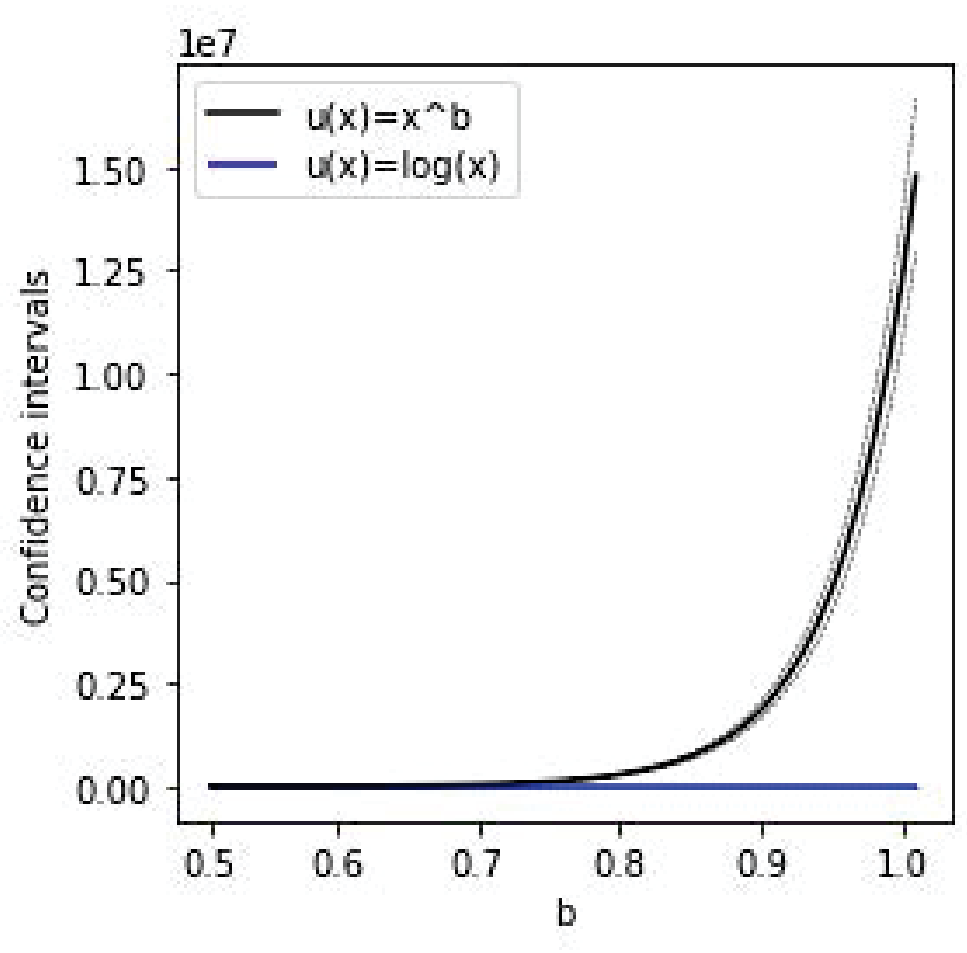}\label{fig15}} 
	\subfloat[Original coordinates]{\includegraphics[width=2.5in,height=2.5in]{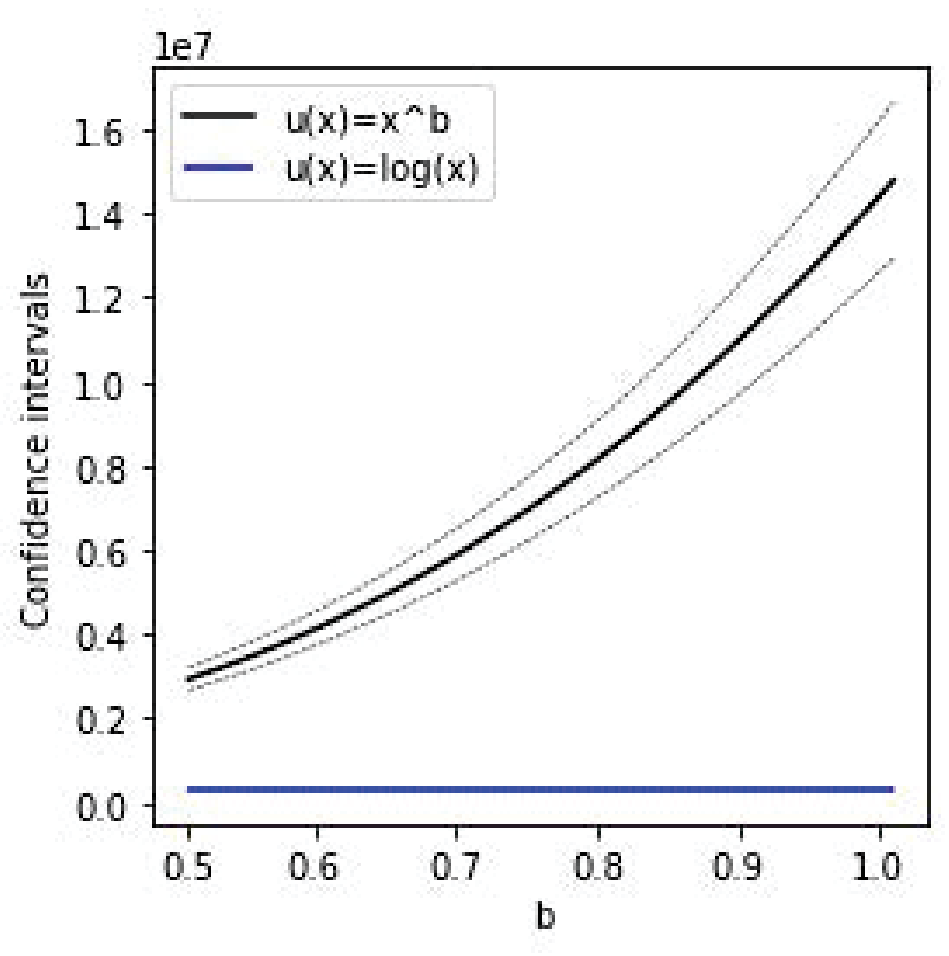}\label{fig16}}
	\caption{Confidence intervals as function of $b$ for the Pareto variable}
	\end{figure}
In this case we see that $u(x)=\ln x$ is a more convenient change of variables than $u_b(x)=x^b$ for any value of $b.$

\section{How to choose the change of variables $u(x)$}
Given a random variable $X$ with infinite mean or infinite variance, the first question to be addressed is whether there exists a function $u(X)$ such that $Y = U(X)$ has a finite mean and a finite variance. To see that it is indeed the case, consider any invertible mapping from R onto a bounded interval.  For example, $u(x)=(1/\pi)\arctan(x)$ or $u(x)=x/(1+|x|),$ map $\bbR$ onto $(-1,1)$ bijectively. Clearly, $E_u[X] = u^{-1}\big(E[u(X)]\big)$ is well defined.

In application, suppose $f_X(x,a)$ denotes the density of $X$ with $a$ denoting some parameter of the density. One then looks for a function $u_b(x)$ such that $u_b(x)f_X(x,a)$ is related to some parametric family of densities, or even better, to a member $f_X (x, g(a, b))$ of the same family for suitably chosen  $g(a,b).$ In this case, the analytical computation of interesting attributes of $u_b(X)$ is straightforward. As an illustration, consider first a positive random variable having a (semi)stable distribution of order $\alpha$ with $0 < \alpha < 1$ (see, e.g., Nolan (2020) or Person and G\'{o}rska (2010)). Denote by $f_\alpha(x)$ the probability density of such a variable. As a family of functions mapping $[0, \infty)$ onto $(0, 1],$ consider $u_b(x) = \exp(-bx)$ with $b > 0.$ It is known that for such a variable $E[\exp(-bX) ] = \exp(-b^{\alpha}) \Rightarrow E[\exp(-2bX) ] = \exp(-2^{\alpha}b^{\alpha}) .$

The best predictor in this case is $E_{u_b}[X]=b^{\alpha-1}$ and the prediction error in the $d_u-$distance is
$$Var(u_b(X)) = e^{-2^{\alpha}b^{\alpha}} - e^{-2b^{\alpha}}.$$
This has a maximum value at 
$$b_e = \bigg(\frac{\ln2-\ln2^{\alpha}}{2 - 2^{\alpha}}\bigg)^{1/\alpha}.$$
The resulting variance is 
$$Var(u_{b_e}) = \bigg(\frac{1}{2^{(1-\alpha)}}\bigg)^{2/(2-2^\alpha)} - \bigg(\frac{1}{2^{(1-\alpha)}}\bigg)^{2^{\alpha}/(2-2^\alpha)}.$$

An example related to the $Student-t$ distributions with parameter $1<\nu\leq2$  goes as follows. If we restrict the variable to $[0,\infty)$ its probability density is given by:
$$f_\nu(x) = \frac{1}{2}\frac{\Gamma(\frac{\nu+1}{2})}{\sqrt{\nu\pi}\Gamma(\frac{\nu}{2})}\bigg(1 + \frac{x^2}{\nu}\bigg)^{-(\nu+1)/2}.$$
For this class of variables the Laplace transform of the density is not available. But note that the change of variables given by $u_b(x)=(1+x^{2}/\nu)^{-b}$ and $b>0.$ It maps $[0,\infty)$ onto $(0,1].$ In this case the best predictor is given by:
$$E_{u_b}[X] = \bigg(\nu\big(1 - \frac{1}{E[u_b(X)]^{1/b}}\big)\bigg)^{1/2}.$$
Next, let $b,c,\nu,\mu$ be positive numbers. Then
$$\int_\bbR\frac{1}{\big(1 + \frac{x^2}{\nu}\big)^b}\frac{1}{\big(1 + \frac{x^2}{\nu}\big)^c}dx = \sqrt{\frac{\nu}{\mu}}\int_\bbR\frac{1}{\big(1 + \frac{x^2}{\mu}\big)^{b+c}}dx$$
where the left-hand side follows from a simple change of scale. Now, if $c=(\nu+1)/2$ is given, we can choose $b$ so that $b+c=(\mu+1)/2.$ As $c=(1+\nu)/2,$ choose $b>0$ so that $\mu>2.$ With all that, $u(b(X)$ has a finite mean and a finite variance. The means and the variances of $u_b(X)$ can then be expressed in terms of quotients of Gamma functions, but we do not pursue the matter further here.

One limitation of this approach is that it depends on knowledge of the distribution function of the random variable whose mean we need to estimate. As the examples in Section 4 suggest, when dealing with real data, we need some reasonable guess of the underlying distribution, because it might be hard to infer that the data comes from a distribution with infinite mean or variance.

An alternative strategy, suggested by the examples, is to consider a family $u_b(x)$ of changes of variables such that $E[u_b(X)]$ and $Var(u_b(X))$ are finite. Then choose a value of the parameter $b$ such that the prediction error measured by $Var(u_b(X)),$ or the confidence interval satisfies some pre-specified criteria.

In general, there is no universal prescription to choose $u(x).$ This is a similar situation regularly encountered by economists, risk managers, investors and other decision-makers when they have to propose a utility function, or other researchers when they chose variables (coordinate systems) to describe a certain phenomenon. In some domains of application, the choices may seem to be more straightforward than in others. For example, sound engineers find it more natural to use decibels to measure the power in a signal than to use Watts. But when it comes to an alternative description of random variables that have infinite means or variances, there does not appear to be a “natural” choice. The quest for a satisfactory change of variables may be a fertile field of research. Meanwhile, a practical approach is to be informed by observations as part of exploratory data analysis.

\section{Concluding remarks}

We proposed a method to transform a random variable with an infinite mean or variance in such a way that the transformed mean or variance becomes finite. The change of variables induces a distance on the class of random variables, in which the notions of best predictor and conditional best predictor can be thought of as a generalized arithmetic mean.  The latter happens to be analogous to the so-called certainty price used in utility theory.

As illustrated through several examples, the change of variables may depend on parameters that can be adjusted so that the prediction error attains an optimal value. We also examined the construction of estimators, and noted that the confidence intervals for the transformed variables can be transformed back to the original coordinates and become confidence intervals for the predictor.

As in insurance and economics, there are two issues to consider: First which transformation u to choose, and, second, how to interpret the best predictor Eu[X]. This again is similar to the problem that arises in insurance with the application of certainty price to risk. The certainty price of the risk depends on the utility function and a suitable utility function must be designed for each situation.  As such, there seems to be no universal prescription for selecting the optimal transformation.   In practice, the search for a convenient change of variables may be executed as part of the exploratory data analysis process.

It is noted that the method that we proposed may be viewed as a transformation exercise, except that the focus of our attention is estimation for the case when the random variables have an infinite mean or infinite variance. A potential use of this method is to compare predictors from possibly different samples. Put in risk management terms, it can be used to compare expected losses, especially in situations where the expected loss is anticipated to be large.

\textbf{Acknowledgments} We thank Joel Cohen for his suggestions and his comments on the manuscript. Victor de la  Pe\~{n}a thanks the CY Initiative of Excellence for the grant “Investissements d’Avenir” ANR-16-IDEX-0008, Project “EcoDep” PSI- AAP2020-0000000013, which funded his work.

\end{document}